\documentclass[12pt,bezier]{article}
\usepackage{times}
\usepackage{booktabs}
\usepackage{pifont}
\usepackage{floatrow}
\usepackage{caption}
\usepackage{mathrsfs}
\usepackage[fleqn]{amsmath}
\usepackage{amsfonts,amsthm,amssymb,mathrsfs,bbding}
\usepackage{txfonts}
\usepackage{graphics,multicol}
\usepackage{graphicx}
\usepackage{color}
\usepackage{amssymb}
\usepackage{caption}
\usepackage{cite}
\usepackage{latexsym,bm}
\usepackage{indentfirst}
\usepackage{color}
\usepackage[colorlinks=true,anchorcolor=blue,filecolor=blue,linkcolor=blue,urlcolor=blue,citecolor=blue]{hyperref}
\usepackage{extarrows}
\usepackage{cite}
\usepackage{latexsym,bm}
\usepackage{mathtools}
\evensidemargin=\oddsidemargin\topmargin=-1.5cm

\newtheorem{thm}{Theorem}[section]
\newtheorem{prob}{Problem}[section]
\newtheorem{claim}{Claim}

\newtheorem{lem}{Lemma}[section]

\theoremstyle{definition}

\addtocounter{section}{0}

\begin{document}
\title{Spectral radius and spanning trees of graphs\footnote{Supported by National Natural Science Foundation of China
{(Nos. 11971445 and 12261032)} and Natural Science Foundation of Henan Province (No. 202300410377).}}
\author{{\bf Guoyan Ao$^{a, b}$}, {\bf Ruifang Liu$^{a}$}\thanks{Corresponding author.
E-mail addresses: aoguoyan@163.com, rfliu@zzu.edu.cn, yuanjj@zzu.edu.cn.},
{\bf Jinjiang Yuan$^{a}$}\\
{\footnotesize $^a$ School of Mathematics and Statistics, Zhengzhou University, Zhengzhou, Henan 450001, China} \\
{\footnotesize $^b$ School of Mathematics and Statistics, Hulunbuir University, Hailar, Inner Mongolia 021008, China}}
\date{}

\date{}
\maketitle
{\flushleft\large\bf Abstract}
For integer $k\geq2,$ a spanning $k$-ended-tree is a spanning tree with at most $k$ leaves.
Motivated by the closure theorem of Broersma and Tuinstra [Independence trees and Hamilton cycles, J. Graph Theory 29 (1998) 227--237],
we provide tight spectral conditions to guarantee the existence of a spanning $k$-ended-tree in a connected graph of order $n$ with extremal graphs being characterized.
Moreover, by adopting Kaneko's theorem [Spanning trees with constraints on the leaf degree, Discrete Appl. Math. 115 (2001) 73--76],
we also present tight spectral conditions for the existence of a spanning tree with leaf degree at most $k$ in a connected graph of order $n$ with extremal graphs being determined, where $k\geq1$ is an integer.

\begin{flushleft}
\textbf{Keywords:} Spanning tree, Spectral radius, Closure, Leaf degree

\end{flushleft}
\textbf{AMS Classification:} 05C50; 05C35

\section{Introduction}
In this paper, we consider simple, undirected and connected graphs.
For undefined terms and notions, one can refer to \cite{Bondy2008} and \cite{Brouwer2011}.
Let $G$ be a graph with vertex set $V(G)$ and edge set $E(G)$.
The order and size of $G$ are denoted by $|V(G)|=n$ and $|E(G)|=e(G)$, respectively.
We denote by $d_{G}(v)$, $\omega(G)$, $i(G)$ and $\overline{G}$ the vertex $v$ of degree in $G$,
the clique number, the number of isolated vertices and the complement of $G,$ respectively.
We use $P_{n}$, $C_{n}$, $K_{n}$ and $K_{m,n}$ to denote the path of order $n$, the cycle of order $n$, the complete graph of order $n$,
and the complete bipartite graph with bipartition $(X,Y)$, where $|X|=m$ and $|Y|=n$.
Let $G_{1}$ and $G_{2}$ be two vertex-disjoint graphs. We denote by $G_{1}+G_{2}$ the disjoint union of $G_{1}$ and $G_{2}.$
The join $G_{1}\vee G_{2}$ is the graph obtained from $G_{1}+G_{2}$ by adding all possible edges between $V(G_1)$ and $V(G_2)$.

Let $A(G)$ and $D(G)$ be the adjacency matrix and diagonal degree matrix of $G$.
Let $Q(G)=D(G)+A(G)$ be the signless Laplacian matrix of $G$.
The largest eigenvalues of $A(G)$ and $Q(G)$, denoted by $\rho(G)$ and $q(G)$,
are called the spectral radius and the signless Laplacian spectral radius of $G$, respectively.
For convenience, Liu, Lai and Das \cite{Liu2019} introduced the nonnegative matrix $A_{a}(G)=aD(G)+A(G)~(a\geq 0)$ of a graph $G$.
The largest eigenvalue of $A_{a}(G)$ is called the $A_{a}$-spectral radius of $G$, denoted by $\rho_{a}(G)$.
It is clear that $\rho_{0}(G)=\rho(G)$ and $\rho_{1}(G)=q(G).$

The concept of closure of a graph was used implicitly by Ore \cite{Ore1960},
and formally introduced by Bondy and Chvatal \cite{Bondy1976}.
Fix an integer $l\geq 0$, the $l$-closure of a graph $G$ is the graph obtained from $G$
by successively joining pairs of nonadjacent vertices whose degree sum is at least $l$
until no such pair exists. Denote by $C_{l}(G)$ the $l$-closure of $G.$

For integer $k\geq2,$ a spanning $k$-ended-tree of a connected graph $G$ is a spanning tree with at most $k$ leaves.
Note that a spanning $2$-ended-tree is a Hamilton path.
Hence a spanning $k$-ended-tree of a graph is a natural generalization of a Hamilton path.
Fiedler and Nikiforov \cite{Fiedler2010} initially proved sufficient conditions for a graph to contain a Hamilton path
in terms of the spectral radius of a graph or its complement.
Ning and Ge \cite{Ning2015}, Li and Ning \cite{Li2016} improved and extended results in \cite{Fiedler2010}.
As a special case of spanning $k$-ended-tree, there are many sufficient conditions to assure a graph to contain a Hamilton path
(see for example \cite{Chvatal1972, Lu2012, Liu2015, Zhou2010, Zhou2017}). In fact, a Hamilton path can also be generalized to a spanning $k$-tree.
A spanning $k$-tree of a connected graph $G$ is a spanning tree in which every vertex has degree at most $k$, where $k\geq2$ is an integer.
Fan et al. \cite{Fan2021} presented spectral conditions for the existence of a spanning $k$-tree in a connected graph.

Broersma and Tuinstra \cite{Broersma1998} showed that the problem of whether a given connected graph contains a spanning $k$-ended tree is NP-complete.
Furthermore, they presented a degree sum condition to guarantee that a connected graph has a spanning $k$-ended tree.

\begin{thm}[Broersma and Tuinstra \cite{Broersma1998}]\label{th1}
Let $k\geq 2$ be an integer, and let $G$ be a connected graph of order $n$.
If $d(u)+d(v)\geq n-k+1$ for every two nonadjacent vertices $u$ and $v$, then $G$ has a spanning $k$-ended tree.
\end{thm}

Moreover, Broersma and Tuinstra \cite{Broersma1998} proved the following closure theorem for the existence of a spanning $k$-ended-tree,
which is crucial to our proof.

\begin{thm}[Broersma and Tuinstra \cite{Broersma1998}]\label{le1}
Let $G$ be a connected graph of order $n,$ and $k$ be an integer with $2\leq k \leq n-1.$
Then $G$ has a spanning $k$-ended-tree if and only if the $(n-1)$-closure $C_{n-1}(G)$ of $G$ has a spanning $k$-ended-tree.
\end{thm}

Inspired by the work of Broersma and Tuinstra \cite{Broersma1998}, we focus on the following interesting problem in this paper.

\begin{prob}
Let $G$ be a connected graph of order $n$ which has no a spanning $k$-ended-tree. What is the maximum spectral radius (or signless Laplacian spectral radius) of $G?$
Moreover, characterize all the extremal graphs.
\end{prob}

For more results on spanning $k$-ended-tree, one can refer to
\cite{Chen2020, Chen2014, Chen2019, Flandrin2008, Kyaw2009, Kyaw2011, Tsugaki2007, win1979}.
Inspired by the ideas on Hamilton path from Li and Ning \cite{Li2016} and using typical spectral techniques,
we prove tight spectral conditions to guarantee the existence of a spanning $k$-ended-tree in a connected graph.
Let $R(n,s)$ be an $s$-regular graph with $n$ vertices.

\begin{thm}\label{main1}
Let $G$ be a connected graph of order $n$ and $k\geq2$ be an integer. Each of the following holds.\\
(i) If $n\geq {\rm max}\{6k+5, k^{2}+\frac{3}{2}k+2\}$ and $\rho(G)\geq \rho(K_{1}\vee (K_{n-k-1}+kK_{1})),$
then $G$ has a spanning $k$-ended-tree unless $G\cong K_{1}\vee(K_{n-k-1}+kK_{1})$.\\
(ii) If $n\geq {\rm max}\{6k+5, \frac{3}{2}k^{2}+\frac{3}{2}k+2\}$ and $q(G)\geq q(K_{1}\vee (K_{n-k-1}+kK_{1})),$
then $G$ has a spanning $k$-ended-tree unless $G\cong K_{1}\vee(K_{n-k-1}+kK_{1})$.\\
(iii) If $\rho(\overline{G})\leq \sqrt{k(n-2)},$ then $G$ has a spanning $k$-ended-tree unless $G\cong K_{1,k+1}$.\\
(iv) If $q(\overline{G})\leq n+k-2$ and $G\notin R(t, t-\frac{n+k}{2})\vee K_{n-t} ~(\frac{n+k}{2}\leq t\leq n),$
then $G$ has a spanning $k$-ended-tree.
\end{thm}

Kaneko \cite{Kaneko2001} introduced the concept of leaf degree of a spanning tree.
Let $T$ be a spanning tree of a connected graph $G$. The leaf degree of a vertex $v\in V(T)$ is defined as the number of leaves adjacent to $v$ in $T$.
Furthermore, the leaf degree of $T$ is the maximum leaf degree among all the vertices of $T$.
They posed a necessary and sufficient condition for a connected graph to contain a spanning tree with leaf degree at most $k$.
Let $i(G-S)$ denote the number of isolated vertices of $G-S.$

\begin{thm}[Kaneko \cite{Kaneko2001}]\label{th2}
Let $G$ be a connected graph and $k\geq 1$ be an integer.
Then $G$ has a spanning tree with leaf degree at most $k$ if and only if
$i(G-S)<(k+1)|S|$ for every nonempty subset $S\subseteq V(G)$.
\end{thm}

In this paper, we consider the following spectral extremal problem.

\begin{prob}
Let $G$ be a connected graph of order $n$ which has no a spanning tree with leaf degree at most $k$. What is the maximum spectral radius (or signless Laplacian spectral radius) of $G?$
Moreover, characterize all the extremal graphs.
\end{prob}

Motivated by Kaneko's theorem, we present tight spectral conditions
for the existence of a spanning tree with leaf degree at most $k$ in a connected graph.

\begin{thm}\label{main2}
Let $G$ be a connected graph of order $n\geq 2k+12$, where $k\geq1$ is an integer and $a\in \{0, 1\}$.
If $\rho_{a}(G)\geq \rho_{a}(K_{1}\vee (K_{n-k-2}+(k+1)K_{1})),$ then $G$ has a spanning tree with leaf degree at most $k$
unless $G\cong K_{1}\vee (K_{n-k-2}+(k+1)K_{1})$.
\end{thm}

\section{Proof of Theorem \ref{main1}}

Before presenting our main result, we first show that an $(n-1)$-closed connected graph $G$ must contain a large clique
if its number of edges is large enough.

\begin{lem}\label{le11}
Let $H$ be an $(n-1)$-closed connected graph of order $n\geq 6k+5$, where $k\geq2$ is an integer.
If $$e(H)\geq {n-k-1\choose 2}+k^{2}+k+1,$$ then $\omega(H)\geq n-k$.
\end{lem}

\begin{proof}
For any two vertices $u,v\in V(H)$ with degree at least $\frac{n-1}{2}$, we have $d_{H}(u)+d_{H}(v)\geq n-1.$ Note that $H$ is an $(n-1)$-closed graph.
Then any two vertices of degree at least $\frac{n-1}{2}$ must be adjacent in $H$.
Let $C$ be the vertex set of a maximum clique of $H$ which contains all vertices of degree at least $\frac{n-1}{2},$
and $F$ be the subgraph of $H$ induced by $V(H)\setminus C.$ Let $|C|=r$. Then $|V(F)|=n-r$.
Next we will evaluate the value of $r.$

\vspace{1.5mm}
\noindent\textbf{Case 1.} $1 \leq r\leq\frac{n}{3}+k+1.$
\vspace{1.5mm}

Note that $$e(H[C])={r\choose 2}, e(C, V(F))=\sum_{u\in V(F)}d_{C}(u)~~\mbox{and}~~e(F)=\frac{\sum_{u\in V(F)}d_{H}(u)-\sum_{u\in V(F)}d_{C}(u)}{2}.$$

\begin{claim}\label{cla1}
{\rm $d_{H}(u)\leq \frac{n-2}{2}$ and $d_{C}(u)\leq r-1$ for each $u\in V(F)$.}
\end{claim}

\begin{proof}
Suppose to the contrary that $d_{H}(u)\geq\frac{n-1}{2}$ or $d_{C}(u)\geq r$ for each $u\in V(F)$.
Assume first that $d_{H}(u)\geq\frac{n-1}{2}$. Then we have $d_{H}(u)+d_{H}(v)\geq n-1$ for each $v\in C$. Note that $H$ is an $(n-1)$-closed graph.
Then $u$ is adjacent to every vertex of $C,$ and hence $C\cup\{u\}$ is a larger clique, which contradicts the maximality of $C$.
Note that $|C|=r$. If $d_{C}(u)\geq r$, then $d_{C}(u)=r,$ and hence $u$ is adjacent to every vertex of $C.$
It follows that $C\cup\{u\}$ is a larger clique, a contradiction.
\end{proof}

By Claim \ref{cla1}, we have
\begin{eqnarray*}
e(H)&=&e(H[C])+e(C,V(F))+e(F)\\
&=&{r\choose 2}+\frac{\sum_{u\in V(F)}d_{C}(u)+\sum_{u\in V(F)}d_{H}(u)}{2}\\
&\leq&{r\choose 2}+\frac{(r-1)(n-r)}{2}+\frac{(n-2)(n-r)}{4}\\
&=&(\frac{n}{4}+\frac{1}{2})r+\frac{n^{2}}{4}-n\\
&\leq&(\frac{n}{4}+\frac{1}{2})(\frac{n}{3}+k+1)+\frac{n^{2}}{4}-n\\
&=&\frac{n^{2}}{3}+(\frac{k}{4}-\frac{7}{12})n+\frac{k}{2}+\frac{1}{2}\\
&<& {n-k-1\choose 2}+k^{2}+k+1,
\end{eqnarray*}
for $n\geq 6k+5$. This contradicts $e(H)\geq {n-k-1\choose 2}+k^{2}+k+1.$

\vspace{1.5mm}
\noindent\textbf{Case 2.} $\frac{n}{3}+k+1 <r\leq n-k-1.$
\vspace{1.5mm}

Note that $$e(H[C])={r\choose 2} ~~\mbox{and}~~ e(C,V(F))+e(F)\leq \sum_{u\in V(F)}d_{H}(u).$$

\begin{claim}\label{cla2}
{\rm $d_{H}(u)\leq n-r-1$ for each $u\in V(F)$.}
\end{claim}

\begin{proof}
Suppose that $d_{H}(u)\geq n-r$ for each $u\in V(F)$.
Then we have $d_{H}(u)+d_{H}(v)\geq (n-r)+(r-1)=n-1$ for each $v\in C$. Note that $H$ is $(n-1)$-closed.
Then $u$ is adjacent to every vertex of $C$. This implies that $C\cup\{u\}$ is a larger clique, a contradiction.
\end{proof}

By Claim \ref{cla2}, we have
\begin{eqnarray*}
e(H)&=&e(H[C])+e(C,V(F))+e(F)\\
&\leq&{r\choose 2}+\sum_{u\in V(F)}d_{H}(u)\\
&\leq&{r\choose 2}+(n-r)(n-r-1)\\
&=&\frac{3}{2}r^{2}-(2n-\frac{1}{2})r+n^{2}-n\\
&\triangleq&f(r).
\end{eqnarray*}
Note that $f(r)$ is a concave function on $r.$
Since $\lfloor\frac{n}{3}\rfloor+k+2 \leq t\leq n-k-1,$ then we have
$$e(H)\leq \max\{f(\lfloor\frac{n}{3}\rfloor+k+2),f(n-k-1)\}={n-k-1\choose 2}+k^{2}+k<e(H),$$
for $n\geq 3k+5$, a contradiction.

By Cases 1 and 2, we know that $r\geq n-k,$ and hence $\omega(H)\geq r\geq n-k$. This completes the proof.
\end{proof}

With the help of Lemma \ref{le11}, we prove a technical sufficient condition in terms of $e(G)$ to assure that $G$ has a spanning $k$-ended-tree.

\begin{lem}\label{le12}
Let $G$ be a connected graph of order $n\geq {\rm max}\{6k+5, k^{2}+k+2\}$, where $k\geq2$ is an integer.
If $$e(G)\geq {n-k-1\choose 2}+k^{2}+k+1,$$ then $G$ has a spanning $k$-ended-tree unless $C_{n-1}(G)\cong K_{1}\vee(K_{n-k-1}+kK_{1})$.
\end{lem}

\begin{proof}
Suppose that $G$ has no a spanning $k$-ended-tree, where $n\geq {\rm max}\{6k+5, k^{2}+k+2\}$ and $k\geq2$.
Let $H=C_{n-1}(G)$. It suffices to prove that $H\cong K_{1}\vee(K_{n-k-1}+kK_{1}).$

By Theorem \ref{le1}, $H$ has no a spanning $k$-ended-tree either.
Note that $G\subseteq H$. Since $e(G)\geq {n-k-1\choose 2}+k^{2}+k+1$, then $e(H)\geq {n-k-1\choose 2}+k^{2}+k+1$.
According to Lemma \ref{le11}, we have $\omega(H)\geq n-k$.

Next we will characterize the structure of $H$.
Let $C$ be a maximum clique of $H$ and $F$ be a subgraph of $H$ induced by $V(H)\backslash C.$
First we claim that $V(F)\neq \emptyset$. If $V(F)= \emptyset$, then $H\cong K_{n}$.
Obviously, $H$ has a spanning $k$-ended-tree, a contradiction.

\begin{claim}\label{cla3}
{\rm $\omega(H)=n-k$.}
\end{claim}

\begin{proof}
Recall that $\omega(H)\geq n-k$. It suffices to prove that $\omega(H)\leq n-k$. Assume to the contrary that $\omega(H)\geq n-k+1$. Let $C'$ be an $(n-k+1)$-clique of $H$ and $F'$ be a subgraph of $H$ induced by $V(H)\backslash C'$. Then $|V(F')|=k-1>0$ because of $k\geq2.$
Note that $G$ is a connected graph.
Then $H$ is also connected, and hence there exists a vertex $v\in V(F')$ which is adjacent to a vertex in $V(C').$
We take a path $P$ such that $V(P)=V(C')\cup \{v\}.$
Note that $|V(F')-v|=k-2$. Therefore, $H$ has a spanning tree with at most $k$ leaves,
which means that $H$ has a spanning $k$-ended-tree, a contradiction.
So we have $\omega(H)=n-k$, as we claimed.
\end{proof}

By Claim \ref{cla3}, we know that $C\cong K_{n-k}$. Let $V(C)= \{u_{1} ,u_{2}, \ldots, u_{n-k}\}$ and $V(F)= \{v_{1} ,v_{2}, \ldots, v_{k}\}$.

\begin{claim}\label{cla4}
{\rm $V(F)$ is an independent set.}
\end{claim}

\begin{proof}
Assume that $v_{i}$ is adjacent to $v_{j}$, where $v_{i}, v_{j}\in V(F)$.
By the connectedness of $H$, there exists a path from $v_{i}$ to $C$.
Thus, we can take the path $P$ such that $V(C)\cup \{v_{i}, v_{j}\} \subseteq V(P)$.
Note that $|V(H)\setminus V(P)|\leq k-2$. Then $H$ has a spanning tree with at most $k$ leaves.
That is, $H$ has a spanning $k$-ended-tree, a contradiction.
\end{proof}

\begin{claim}\label{cla5}
{\rm $d_{H}(v_{i})=1$ for every $v_{i}\in V(F).$}
\end{claim}

\begin{proof}
Suppose there exists a vertex $v_{i}\in V(F)$ with $d_{H}(v_{i})\geq 2.$ According to Claim \ref{cla4},
without loss of generality, we can assume that $v_{i}$ is adjacent to $u_{1}$ and $u_{2}$, where $u_{1}, u_{2}\in V(C).$
We take the path $P=u_{1}v_{i}u_{2}u_{3}\cdots u_{n-k}$ of length $n-k$.

Furthermore, we claim that $v_{j}~(j\neq i)$ must be adjacent to $u_{2}$ for each $v_{j}\in V(F)$.
In fact, if $v_{j}$ is adjacent to $u_{1}$ or $u_{n-k}$, then there exists a path $P'=P+v_{j}u_{1}$ or $P+u_{n-k}v_{j}$ of length $n-k+1$.
If $v_{j}~(j\neq i)$ is adjacent to $u_{s}$, where $3\leq s\leq n-k-1$. Then there exists a path
$P'=P-u_{s-1}u_{s}+u_{s-1}u_{n-k}+u_{s}v_{j}$ of length $n-k+1$. By the above proof, we can always find a path $P'$ with $V(P')=V(C)\cup \{v_{i}, v_{j}\}.$
Note that $|V(F)-v_{i}-v_{j}|=k-2$.
Hence $H$ has a spanning tree with at most $k$ leaves.
This implies that $H$ has a spanning $k$-ended-tree, a contradiction. Hence $v_{j}~(j\neq i)$ must be adjacent to $u_{2}$ for every $v_{j}\in V(F)$.

At this time, we can observe that $H$ has a spanning tree $T=P-u_{2}u_{3}+u_{1}u_{n-k}+\sum _{j\neq i}u_{2}v_{j}$ such that $L(T)= V(F)\setminus\{v_{i}\}\cup \{u_{3}\}$, where $L(T)$ denotes the set of leaves of $T.$ Note that $|L(T)|=k$.
Hence $H$ has a spanning $k$-ended-tree, a contradiction.
\end{proof}

\begin{claim}\label{cla6}
{\rm $N_{H}(v_{i})\cap C=N_{H}(v_{j})\cap C,$ where $i\neq j$.}
\end{claim}

\begin{proof}
By Claims \ref{cla4} and \ref{cla5}, we assume that $N_{H}(v_{i})\cap C=u_{i}$ and $N_{H}(v_{j})\cap C=u_{j},$ where $i\neq j$.
Then there exists a path $P$ such that $V(P)= V(C)\cup \{v_{i}, v_{j}\}$ and $L(P)=\{v_{i}, v_{j}\}$, where $L(P)$ is the set of leaves of $P.$
Note that $|V(F)-v_{i}- v_{j}|=k-2$.
Note that $H$ is a connected graph. Then $H$ has a spanning $k$-ended-tree, a contradiction.
Hence $N_{H}(v_{i})\cap C=N_{H}(v_{j})\cap C,$  where $i\neq j$.
\end{proof}

\begin{figure}
\centering
\includegraphics[width=0.25\textwidth]{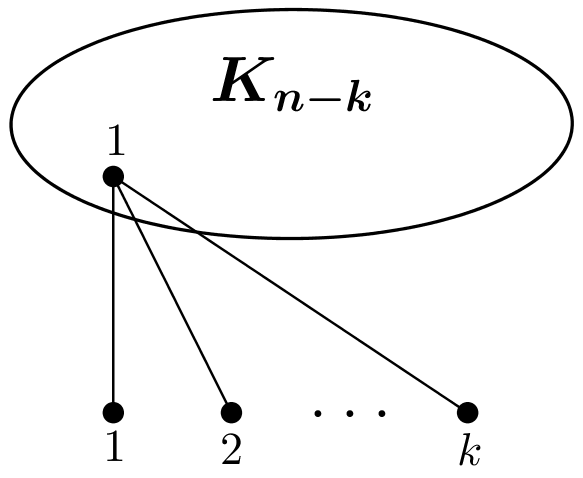}\\
\caption{Graph $K_{1}\vee (K_{n-k-1}+kK_{1}).$
}\label{fig1}
\end{figure}

By the above claims, we know that $H\cong K_{1}\vee (K_{n-k-1}+kK_{1})$ (see Fig. \ref{fig1}).
Note that the vertices of $kK_{1}$ are only adjacent to a vertex of $K_{n-k}$.
Then the number of leaves of a spanning tree of $K_{1}\vee (K_{n-k-1}+kK_{1})$ is at least $k+1$.
This implies that $K_{1}\vee (K_{n-k-1}+kK_{1})$ has no a spanning $k$-ended-tree.
Note that $e(H)={n-k\choose 2}+k$ and $n\geq k^{2}+k+2$. Then we have
$e(H)\geq {n-k-1\choose 2}+k^{2}+k+1.$
Therefore, $H=C_{n-1}(G)\cong K_{1}\vee (K_{n-k-1}+kK_{1}),$ as desired.
\end{proof}

Next we will present important upper and lower bounds of $\rho(G)$ and $q(G)$ that will be used in our subsequent arguments.

\begin{lem}[Hong \cite{Hong1988}]\label{le2}
Let $G$ be a connected graph with $n$ vertices. Then
$$\rho(G)\leq \sqrt{2e(G)-n+1}.$$
\end{lem}

\begin{lem}[Li and Ning \cite{Li2016}]\label{le3}
Let $G$ be a graph with non-empty edge set. Then
$$\rho(G)\geq {\rm min}\{\sqrt{d(u)d(v)}:uv\in E(G)\}.$$
Moreover, if $G$ is connected, then equality holds if and only if $G$ is regular or semi-regular bipartite.
\end{lem}

\begin{lem}[Das \cite{Das2004}, Feng and Yu \cite{Feng2009}]\label{le4}
Let $G$ be a connected graph on $n$ vertices and $e(G)$ edges. Then
$$q(G)\leq \frac{2e(G)}{n-1}+n-2.$$
\end{lem}

\begin{lem}[Li and Ning \cite{Li2016}]\label{le6}
Let $G$ be a graph with non-empty edge set. Then
$$q(G)\geq {\rm min}\{d(u)+d(v):uv\in E(G)\}.$$
Moreover, if $G$ is connected, then equality holds if and only if $G$ is regular or semi-regular bipartite.
\end{lem}

Let $A=(a_{ij})$ and $B=(b_{ij})$ be two $n\times n$ matrices.
Define $A\leq B$ if $a_{ij}\leq b_{ij}$ for all $i$ and $j$, and define $A< B$ if $A\leq B$ and $A\neq B$.

\begin{lem}[Berman and Plemmons \cite{Berman1979}, Horn and Johnson \cite{Horn1986}]\label{le5}
Let $A=(a_{ij})$ and $B=(b_{ij})$ be two $n\times n$ matrices with the spectral radii $\lambda(A)$ and $\lambda(B)$.
If $0\leq A\leq B$, then $\lambda(A)\leq \lambda(B)$.
Furthermore, if $B$ is irreducible and $0\leq A < B$, then $\lambda(A)<\lambda(B)$.
\end{lem}

Now, we are in a position to present the proof of Theorem \ref{main1}.

\medskip
\noindent  \textbf{Proof of Theorem \ref{main1}.}
(i) Let $G$ be a connected graph of order $n.$ Suppose to the contrary that $G$ has no a spanning $k$-ended-tree, where $n\geq {\rm max}\{6k+5, k^{2}+\frac{3}{2}k+2\}$ and $k\geq2.$
By Lemma \ref{le5}, we know that $\rho(K_{1}\vee (K_{n-k-1}+kK_{1}))>\rho(K_{n-k})=n-k-1$.
By Lemma \ref{le2} and the assumption of Theorem \ref{main1}, we have
$$n-k-1<\rho(K_{1}\vee (K_{n-k-1}+kK_{1})) \leq \rho(G)\leq \sqrt{2e(G)-n+1}.$$
Hence
$e(G)>\frac{(n-k-1)^{2}+n-1}{2}\geq {n-k-1\choose 2}+k^{2}+k+1$
for $n\geq k^{2}+\frac{3}{2}k+2$. Let $H=C_{n-1}(G)$.
By Lemma \ref{le12}, we have $H\cong K_{1}\vee (K_{n-k-1}+kK_{1}).$ Note that $G\subseteq H.$
Then we have $$\rho(G)\leq\rho(H)=\rho(K_{1}\vee (K_{n-k-1}+kK_{1})).$$
Combining the assumption $\rho(G)\geq \rho(K_{1}\vee (K_{n-k-1}+kK_{1}))$, we have $G\cong H.$
From the end of the proof in Lemma \ref{le12}, we know that $K_{1}\vee (K_{n-k-1}+kK_{1})$ has no a spanning $k$-ended-tree.
Hence $G\cong K_{1}\vee (K_{n-k-1}+kK_{1}),$ as desired.
\medskip

(ii) Let $G$ be a connected graph of order $n$ with $n\geq {\rm max}\{6k+5, \frac{3}{2}k^{2}+\frac{3}{2}k+2\}.$ Assume that $G$ has no a spanning $k$-ended-tree, where $k\geq2$ is an integer. By Lemma \ref{le5}, we have $q(K_{1}\vee (K_{n-k-1}+kK_{1}))>q(K_{n-k})=2(n-k-1).$
By Lemma \ref{le4} and the assumption, we have
  $$2(n-k-1)<q(K_{1}\vee (K_{n-k-1}+kK_{1})) \leq q(G)\leq \frac{2e(G)}{n-1}+n-2.$$
We can deduce that
$e(G)> \frac{[2(n-k-1)-n+2](n-1)}{2}\geq {n-k-1\choose 2}+k^{2}+k+1$
for $n\geq \frac{3}{2}k^{2}+\frac{3}{2}k+2$. Let $H=C_{n-1}(G)$.
By Lemma \ref{le12}, we have $H\cong K_{1}\vee (K_{n-k-1}+kK_{1})$.
Since $G\subseteq H$, then
$$q(G)\leq q(H)=q(K_{1}\vee (K_{n-k-1}+kK_{1})).$$
By the assumption, $q(G)\geq q(K_{1}\vee (K_{n-k-1}+kK_{1}))$, and hence $G\cong H$.
Recall that $K_{1}\vee (K_{n-k-1}+kK_{1})$ has no a spanning $k$-ended-tree.
Hence $G\cong K_{1}\vee (K_{n-k-1}+kK_{1}).$
\medskip

(iii) Let $G$ be a connected graph of order $n$ and $k\geq2$ be an integer. Let $H=C_{n-1}(G).$ If $H\cong K_{n},$ then $H$ has a spanning $k$-ended-tree.
By Theorem \ref{le1}, $G$ also has a spanning $k$-ended-tree, and the result follows.
Next we always assume that $H\ncong K_{n}.$

Suppose to the contrary that $G$ has no a spanning $k$-ended-tree.
By Theorem \ref{le1}, $H$ has no a spanning $k$-ended-tree either.
By Theorem \ref{th1}, we have
$$d_{H}(u)+d_{H}(v)\leq n-k$$
for every pair of nonadjacent vertices $u$ and $v$ (always exists) of $H$. Then we have
$$d_{\overline{H}}(u)+d_{\overline{H}}(v)=n-1-d_{H}(u)+n-1-d_{H}(v)\geq n+k-2$$
for every edge $uv\in E(\overline{H})$.
Notice that every non-trivial component of $\overline{H}$ has a vertex with the degree at least $\frac{n+k-2}{2}.$
Then its order is at least $\frac{n+k}{2}$, which implies that $\overline{H}$ has exactly one non-trivial component $F$
and $V(\overline{H}\backslash F)$ is the set of isolated vertices.
Hence $$d_{F}(u)+d_{F}(v)=d_{\overline{H}}(u)+d_{\overline{H}}(v)\geq n+k-2$$ for $uv\in E(F).$
Since $d_{H}(u)\geq d_{G}(u)\geq 1$ and $d_{H}(v)\geq d_{G}(v)\geq 1$, then we have $d_{F}(u)=d_{\overline{H}}(u)\leq n-2$ and $d_{F}(v)=d_{\overline{H}}(v)\leq n-2$.
Hence $d_{F}(u)\geq n+k-2-d_{F}(v)\geq k$ and $d_{F}(v)\geq n+k-2-d_{F}(u)\geq k.$
That is, $k \leq d_{F}(u)\leq n-2$ and $k \leq d_{F}(v)\leq n-2$.
It is easy to see that
 $$d_{F}(u)d_{F}(v)\geq d_{F}(u)(n+k-2-d_{F}(u))\triangleq f(d_{F}(u)).$$
Note that $f(d_{F}(u))$ is a convex function on $d_{F}(u)\in [k, n-2]$. Then
$$d_{\overline{H}}(u)d_{\overline{H}}(v)=d_{F}(u)d_{F}(v)\geq \mbox{min}\{f(k),f(n-2)\}=k(n-2).$$
for every edge $uv\in E(\overline{H})$.
By the assumption and Lemma \ref{le3}, we have
$$\sqrt{k(n-2)}\geq \rho(\overline{G}) \geq \rho(\overline{H})\geq \mbox{min}_{uv\in E(\overline{H})}\sqrt{d_{\overline{H}}(u)d_{\overline{H}}(v)}\geq\sqrt{k(n-2)}.$$
Then all the above inequalities must be equalities.
This implies that $\overline{G}\cong\overline{H}$, $\rho(\overline{H})=\sqrt{k(n-2)},$ and there exists an edge $uv\in E(\overline{H})$ such that
$d_{\overline{H}}(u)=k$ and $d_{\overline{H}}(v)= n-2$.
Note that $F$ is a unique non-trivial component of $\overline{H}.$ Then
$$\rho(F)=\rho(\overline{H})=\mbox{min}_{uv\in E(\overline{H})}\sqrt{d_{\overline{H}}(u)d_{\overline{H}}(v)}=
\mbox{min}_{uv\in E(F)}\sqrt{d_{F}(u)d_{F}(v)}.$$
By Lemma \ref{le3}, $F$ is regular or semi-regular bipartite.

Assume first that $F$ is semi-regular bipartite. By symmetry, we can assume that $d_{F}(u)=k$ and $d_{F}(v)=n-2$ for every edge $uv\in E(F)$. Hence
$F\cong K_{k,n-2}.$ It is easy to see that $n\geq k+(n-2),$ and thus $k=2$ since $k\geq 2$.
Then we have $\overline{H}=F\cong K_{2,n-2}$, and hence $G\cong H\cong K_{2}+K_{n-2}$,
which contradicts that $G$ is connected.

Hence $F$ is regular. Then $d_{F}(v)=k=n-2$ for each $v\in V(F),$ and hence $n=k+2$.
If $V(\overline{H}\backslash F)=\phi$, then $\overline{H}=F$ is an $(n-2)$-regular graph with $n$ vertices.
Obviously, $\overline{H}$ is obtained from $K_{n}$ by deleting a perfect matching of $\overline{H}$.
Then $G\cong H$ is the perfect matching of $G,$ which contradicts the connectedness of $G.$
Next we consider $V(\overline{H}\backslash F)\neq \phi$. Then $\overline{H}=F+K_{1},$ where $F\cong K_{k+1}$.
So we have $G\cong H\cong K_{1,k+1}$. It is obvious that $G$ has no a spanning $k$-ended-tree. Moveover, $\rho(\overline{G})=\rho(K_{k+1})=k=\sqrt{k(n-2)}.$
Hence $G\cong K_{1,k+1}$.

(iv) Let $H=C_{n-1}(G).$ By the beginning of the proof in (iii), we still assume that $H\ncong K_{n}.$ Suppose that $G$ has no a spanning $k$-ended-tree.
By Theorem \ref{le1}, $H$ has no a spanning $k$-ended-tree either. From the proof of (iii),
we can obtain that $\overline{H}$ has exactly one non-trivial component $F$ and
$$d_{F}(u)+d_{F}(v)\geq n+k-2$$
for each $uv\in E(F)$, where $k \leq d_{F}(u)\leq n-2$ and $k \leq d_{F}(v)\leq n-2$.
By the assumption and Lemma \ref{le6}, we have
$$n+k-2\geq q(\overline{G}) \geq q(\overline{H})=q(F)\geq \mbox{min}_{uv\in E(F)}(d_{F}(u)+d_{F}(v))\geq n+k-2.$$
Then all the above inequalities must be equalities, which implies that
$\overline{G}\cong\overline{H}$, $q(F)=n+k-2,$ and there exists an edge $uv\in E(F)$ such that $d_{F}(u)+d_{F}(v)= n+k-2$.
Furthermore, we have $q(F)= \mbox{min}_{uv\in E(F)}(d_{F}(u)+d_{F}(v))$.
By Lemma \ref{le6}, $F$ is regular or semi-regular bipartite.

If $F$ is semi-regular bipartite, then $F\cong K_{t, n+k-2-t}$ with $k\leq t\leq n-2.$
Since $n\geq t+(n+k-2-t)$ and $k\geq 2$, then we have $k=2$. Hence $\overline{H}=F\cong K_{t,n-t},$ where $k \leq t\leq n-2$.
Then $G\cong H\cong K_{t}+K_{n-t}$ with $k \leq t\leq n-2$, which contradicts the connectedness of $G$.

Hence $F$ is regular. Then $F$ is an $\frac{n+k-2}{2}$-regular graph with $t$ vertices, where $\frac{n+k}{2}\leq t\leq n$.
It is easy to see that $\overline{H}=F+(n-t)K_{1}$, and hence $G\cong H\cong \overline{F}\vee K_{n-t}$.
Note that $\overline{F}$ is an $(t-\frac{n+k}{2})$-regular graph with $t$ vertices.
For convenience, let $\overline{F}=R(t, t-\frac{n+k}{2})$. Then we can write
$$G\cong R(t, t-\frac{n+k}{2})\vee K_{n-t},$$
where $\frac{n+k}{2}\leq t\leq n,$ a contradiction.
The proof is completed. \hspace*{\fill}$\Box$

\section{Proof of Theorem \ref{main2}}

Recall the nonnegative matrix $A_{a}(G)=aD(G)+A(G)~(a\geq 0)$ of a graph $G$.
Let $\rho_{a}(G)$ be the $A_{a}$-spectral radius of $G.$ Obviously, $\rho_{0}(G)=\rho(G)$ and $\rho_{1}(G)=q(G).$
Now, we are ready to present the proof of Theorem \ref{main2}.

\medskip
\noindent  \textbf{Proof of Theorem \ref{main2}.}
Suppose to the contrary that $G$ has no a spanning tree with leaf degree at most $k$ for $n\geq 2k+12$ and $k\geq 1$.
By Theorem \ref{th2}, there exists some nonempty subset $S\subseteq V(G)$ such that $i(G-S)\geq (k+1)|S|.$
Let $|S|=s $. Then $G$ is a spanning subgraph of $G'=K_{s}\vee(K_{n-(k+2)s}+(k+1)sK_{1})$ (see Fig. \ref{fig2}).
By Lemma \ref{le5}, we have $\rho_{a}(G)\leq \rho_{a}(G')$ for $a\in \{0, 1\}$.
We distinguish the proof into the following two cases.

\vspace{1.5mm}
\noindent\textbf{Case 1.}  $s=1.$
\vspace{1mm}

Then $G'\cong K_{1}\vee(K_{n-k-2}+(k+1)K_{1}).$ From the above, we know that $$\rho_{a}(G)\leq \rho_{a}(K_{1}\vee(K_{n-k-2}+(k+1)K_{1})).$$
By the assumption $\rho_{a}(G)\geq \rho_{a}(K_{1}\vee(K_{n-k-2}+(k+1)K_{1}))$, then we have $G\cong K_{1}\vee(K_{n-k-2}+(k+1)K_{1}).$
Note that the vertices of $(k+1)K_{1}$ are only adjacent to a vertex of $K_{n-k-1}$.
Then the leaf degree of any spanning tree of $K_{1}\vee(K_{n-k-2}+(k+1)K_{1})$ is at least $k+1$.
This implies that $K_{1}\vee(K_{n-k-2}+(k+1)K_{1})$ has no a spanning tree with leaf degree at most $k$.
Hence $G\cong K_{1}\vee(K_{n-k-2}+(k+1)K_{1})$.

\vspace{1.5mm}
\noindent\textbf{Case 2.}  $s\geq 2.$
\vspace{1mm}

Note that $G'=K_{s}\vee(K_{n-(k+2)s}+(k+1)sK_{1})$ and $e(G')={n-(k+1)s\choose 2}+(k+1)s^{2}$.
Next we will discuss the different values of $a$.

\vspace{1.5mm}
\noindent\textbf{Subcase 2.1.}  $a=0.$
\vspace{1mm}

By Lemma \ref{le2}, we have
\begin{eqnarray*}
\rho_{0}(G')&\leq& \sqrt{2e(G')-n+1} \\
&=&\sqrt{(n-ks-s)(n-ks-s-1)+2(k+1)s^{2}-n+1}\\
&=&\sqrt{(k^{2}+4k+3)s^{2}-(2kn+2n-k-1)s+n^{2}-2n+1}\\
&\triangleq&\sqrt{f_{1}(s)}.
\end{eqnarray*}
Note that $n\geq s+(k+1)s=(k+2)s$. Then $2 \leq s\leq \frac{n}{k+2}.$ Moreover, we claim that $$\mbox{max}_{2 \leq s\leq \frac{n}{k+2}} f_{1}(s)=f_{1}(2).$$
In fact, let $g(n)=f_{1}(2)-f_{1}(\frac{n}{k+2})$, by a direct calculation, we have
$$g(n)= \frac{(k+1)^{2}n^{2}-(4k^{3}+21k^{2}+35k+18)n+4k^{4}+34k^{3}+102k^{2}+128k+56}{(k+2)^{2}}.$$
Note that $g(n)$ is a monotonically increasing function on $n\in[2k+12, +\infty)$. Then
$$f_{1}(2)-f_{1}(\frac{n}{k+2})=g(n)\geq g(2k+12)=\frac{24k^{2}+8k-16}{(k+2)^{2}}>0.$$
This implies that $\mbox{max}_{2 \leq s\leq \frac{n}{k+2}} f_{1}(s)=f_{1}(2).$
By the above proof, we have
\begin{eqnarray*}
\rho_{0}(G)\leq \rho_{0}(G')&\leq& \sqrt{f_{1}(2)} \\
&=&\sqrt{(n-k-2)^{2}-(2kn+2n-3k^{2}-14k-11)}\\
&\leq&\sqrt{(n-k-2)^{2}-(k^{2}+14k+13)}\\
&<&n-k-2,
\end{eqnarray*}
since $n\geq 2k+12$.
On the other hand, by the assumption and Lemma \ref{le5}, we have
$$\rho_{0}(G)\geq \rho_{0}(K_{1}\vee(K_{n-k-2}+(k+1)K_{1}))>\rho_{0}(K_{n-k-1})=n-k-2,$$ a contradiction.

\begin{figure}
\centering
\includegraphics[width=0.3\textwidth]{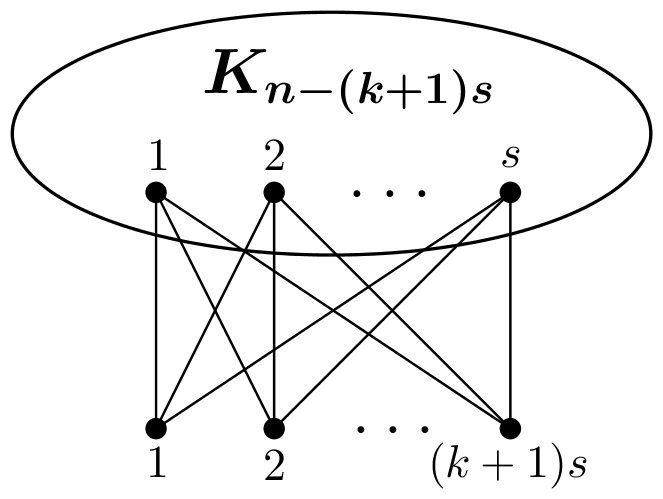}\\
\caption{Graph $K_{s}\vee(K_{n-(k+2)s}+(k+1)sK_{1})$.
}\label{fig2}
\end{figure}

\vspace{1.5mm}
\noindent\textbf{Subcase 2.2.}  $a=1.$
\vspace{1mm}

By Lemma \ref{le4}, we have
\begin{eqnarray*}
\rho_{1}(G')&\leq& \frac{2e(G')}{n-1}+n-2 \\
&=&\frac{(n-ks-s)(n-ks-s-1)+2(k+1)s^{2}}{n-1}+n-2\\
&=&\frac{(k^{2}+4k+3)s^{2}-(2kn+2n-k-1)s+2n^{2}-4n+2}{n-1}\\
&\triangleq& \frac{f_{2}(s)}{n-1}.
\end{eqnarray*}
Note that $2 \leq s\leq \frac{n}{k+2}$.
Similar to the proof of Subcase 2.1, we can obtain that $$\mbox{max}_{2 \leq s\leq \frac{n}{k+2}} f_{2}(s)=f_{2}(2).$$
Then
\begin{eqnarray*}
\rho_{1}(G)\leq \rho_{1}(G')&\leq&\frac{ f_{2}(2)}{n-1} \\
&=&2(n-k-2)-\frac{(2k+2)n-4k^{2}-16k-12}{n-1}\\
&\leq&2(n-k-2)-\frac{12k+12}{n-1}\\
&<&2(n-k-2),
\end{eqnarray*}
since $n\geq 2k+12$. However, by the assumption, we have
$$\rho_{1}(G)\geq \rho_{1}(K_{1}\vee(K_{n-k-2}+(k+1)K_{1}))>\rho_{1}(K_{n-k-1})=2(n-k-2),$$ a contradiction.
We complete the proof.
\hspace*{\fill}$\Box$

\vspace{5mm}
\noindent
{\bf Declaration of competing interest}
\vspace{3mm}

The authors declare that they have no known competing financial interests or personal relationships that could have appeared to influence the work reported in this paper.



\end{document}